\documentclass{amsproc}%
\usepackage{amsfonts}
\usepackage{amsmath}
\usepackage{amssymb}
\usepackage{amscd}
\usepackage{graphicx}
\usepackage{hyperref}
\usepackage{tkz-graph}%
\providecommand{\U}[1]{\protect\rule{.1in}{.1in}}
\theoremstyle{plain}
\newcommand{\q}[1]{``#1''}

\newtheorem{corollary}{Corollary}

\newtheorem{definition}{Definition}
\newtheorem{example}{Example}

\newtheorem{lemma}{Lemma}

\newtheorem{proposition}{Proposition}

\newtheorem{theorem}{Theorem}
\numberwithin{equation}{section}
\begin{document}
\title[{\normalsize ON WEAKLY 1-ABSORBING PRIME IDEALS}]{{\normalsize ON WEAKLY 1-ABSORBING PRIME IDEALS}}

\author{Suat Ko\c{c}}
\address{Department of Mathematics, Marmara University, Istanbul, Turkey.}
\email{suat.koc@marmara.edu.tr}
\author{\"{U}nsal Tekir}
\address{Department of Mathematics, Marmara University, Istanbul, Turkey.}
\email{utekir@marmara.edu.tr}
\author{Eda Y{\i}ld{\i}z}
\address{Department of Mathematics, Yildiz Technical University, Istanbul, Turkey.}
\email{edyildiz@yildiz.edu.tr}

\subjclass[2000]{13A15, 13C05, 54C35}
\keywords{weakly prime ideal, 1-absorbing prime ideal, weakly 2-absorbing ideal, weakly
1-absorbing prime ideal, trivial extension, rings of continuous functions.}

\begin{abstract}
This paper introduce and study weakly 1-absorbing prime ideals in
commutative rings. Let $A\ $be a commutative ring with a nonzero identity
$1\neq0.$\ A proper ideal $P\ $of $A\ $is said to be a weakly 1-absorbing
prime ideal if for each nonunits $x,y,z\in A\ $with $0\neq xyz\in P,\ $then
either $xy\in P$ or $z\in P.\ $In addition to give many properties and
characterizations of weakly 1-absorbing prime ideals, we also determine rings
in which every proper ideal is weakly 1-absorbing prime. Furthermore, we
investigate weakly 1-absorbing prime ideals in $C(X)$, which is the ring of continuous
functions of a topological space $X.\ $

\end{abstract}
\maketitle

\section{Introduction}

Throughout the paper, we focus only on commutative rings with a nonzero
identity and nonzero unital modules. Let $A\ $will always denote such a ring
and $M\ $denote such an $A$-module. The concept of prime ideals and its
generalizations have a significiant place in commutative algebra since they
are used in understanding the structure of rings. Recall that a proper ideal
$P\ $of $A\ $is said to be a \textit{prime ideal }if whenever $xy\in P$ for
some $x,y\in A,\ $then either $x\in P$ or $y\in P$ \cite{AtMac}. The set of
all prime ideals and all maximal ideals will be denoted by $Spec(A)\ $and
$Max(A),\ $respectively. Also, $reg(A)\ $and $u(A)\ $always denote the set of
all regular elements and the set of all units in $A.\ $The importance of prime
ideals led many researchers to work prime ideals and its generalizations. See,
for example, \cite{Ba2}, \cite{Be} and \cite{KoUlTe}. In 2002, Anderson and
Smith in \cite{AnSmi} defined weakly prime ideals which is a generalization of
prime ideals and they used it to study factorization in commutative rings with
zero divisors. A proper ideal $P$ of $A\ $is said to be a weakly prime if
$0\neq xy\in P\ $for each $x,y\in A\ $implies either $x\in P$ or $y\in P.\ $It
is clear that every prime ideal is weakly prime but the converse is not true
in general. For instance, let $k\ $be a field and $A=k[X,Y]/(X^{2}%
,XY,Y^{2}).\ $Then $P=(X,Y^{2})$ is a nonzero weakly prime that is not prime.
Afterwards, Badawi,in his celebrated paper \cite{Ba1}, introduced the notion of
2-absorbing ideals and used them to characterize Dedekind domains. Recall from
\cite{Ba1}, that a nonzero proper ideal $P$ of $A\ $is said to be a
\textit{2-absorbing ideal} if $xyz\in P$ implies either $xy\in P$ or $xz\in P$
or $yz\in P$ for each $x,y,z\in A.\ $Note that every prime ideal is also a
2-absorbing ideal and $P=6%
\mathbb{Z}
$ is an example of 2-absorbing ideal that is not prime in $\mathbb{Z}$, which is the ring of integers. 
The notion of 2-absorbing ideal has been focus of attention
for many researchers, and so many generalizations of 2-absorbing ideals have
been studied. See, for example, \cite{BaTeYe}, \cite{TeKoOrSh} and
\cite{KoUrTe}. Afterwards, Badawi and Darani in \cite{BaDa} defined and
studied weakly 2-absorbing ideals which is a generalization of weakly prime
ideals. A proper ideal $P\ $of $A\ $is said to be a weakly 2-absorbing ideal
if for each $x,y,z\in A\ $with $0\neq xyz\in P,\ $then either we have $xy\in
P\ $or $xz\in P$ or $yz\in P.\ $Recently, Yassine et al. defined a new class
of ideals, which is an intermediate class of ideals between prime ideals and
2-absorbing ideals. Recall from \cite{YasNik} that a proper ideal $P$ of
$A\ $is said to be a \textit{1-absorbing prime ideal} if for each nonunits
$x,y,z\in A\ $with $xyz\in P$, then either $xy\in P$ or $z\in P.\ $Note that
every prime ideal is 1-absorbing prime and every 1-absorbing prime ideal is
2-absorbing ideal. The converses are not true. For instance, $P=6\mathbb{Z}$
 is a 2-absorbing ideal of $\mathbb{Z}$ but not a 1-absorbing prime ideal and also $P=(\overline{0})$ is a
1-absorbing prime ideal of $\mathbb{Z}_4$ which is not prime. Our aim in this paper is to introduce and study
weakly 1-absorbing prime ideal. A proper ideal $P\ $of $A\ $is called a
\textit{weakly 1-absorbing prime ideal} if for each nonunits $x,y,z\in A$ with
$0\neq xyz\in P,\ $then either $xy\in P$ or $z\in P.\ $Among other results in
this paper, in Section 2, we investigate the relations between weakly
1-absorbing prime ideal and other classical ideals such as weakly prime
ideals, weakly 2-absorbing ideals, 1-absorbing prime ideals (See, Example
\ref{exw}-\ref{ex6}).\ Also, we investigate the behaviour of weakly
1-absorbing prime ideals under homomorphisms, in factor rings, in rings of
fractions, in trivial extension $A\ltimes M,$ in cartesian product of rings
(See, Theorem \ref{thom}, Theorem \ref{tfac}, Theorem \ref{tloc}, Theorem
\ref{ttri} and Theorem \ref{tcar}). We give various counter examples associated with the stability of weakly 1-absorbing prime ideals in these algebraic structures 
(See, Example \ref{exco1} and Example \ref{exco2}). In Section 3, we
investigate the rings over which all proper ideals are weakly 1-absorbing
prime. We show that, in Proposition \ref{pc1}, if $A\ $is a ring whose all
proper ideals are weakly 1-absorbing prime, then either $Jac(A)^{2}=0$ or
$Jac(A)=(0:Jac(A)^{2}),$ where $Jac(A)$ is the Jacobson radical of $A$. By
using this result, we characterize rings whose all proper ideals are weakly
1-absorbing prime. In particular we prove that every proper ideal of a ring
$A\ $is weakly 1-absorbing prime if and only if either $(A,\mathfrak{m})\ $is
quasi-local with $\mathfrak{m}^{3}=(0)$ or $A=F_{1}\times F_{2},$ where
$F_{1}$ and $F_{2}$ are fields (See, Theorem \ref{tring}). We dedicate the Section 4 to the study of weakly 1-absorbing prime ideals and weakly prime
ideals in $C(X)$, which is the ring of real valued continuous functions on a topological
space $X.\ $We show that weakly prime $z$-ideals and weakly 1-absorbing
$z$-ideals are coincide in $C(X)\ $(See, Theorem \ref{tcon}).
\section{Characterization of weakly 1-absorbing prime ideals}

\begin{definition}
Let $A\ $be a ring and $P$ be a proper ideal of $A.\ P\ $is said to be a weakly
1-absorbing prime ideal if $0\neq xyz\in P\ $for some nonunits $x,y,z \in A$, then
either $xy\in P\ $or $z\in P.\ $
\end{definition}

\begin{example}
\label{exw}Every weakly prime ideal is also a weakly 1-absorbing prime ideal.
\end{example}

\begin{example}
\textbf{(Weakly 1-absorbing prime ideal that is not weakly prime)} Let $A=%
\mathbb{Z}
_{12}\ $and $P=(\overline{4}).\ $Since $\overline{0}\neq\overline{2}%
.\overline{2}\in P\ $and $\overline{2}\notin P,\ P\ $is not a weakly prime
ideal of $A.\ $Now, we will show that $P\ $is a weakly 1-absorbing prime ideal
of $A.\ $To see this, let $\overline{0}\neq\overline{x}\overline{y}%
\overline{z}\in P$ for some nonunits $\overline{x},\overline{y},\overline
{z}\ $in $A.\ $Assume that $\overline{x}\overline{y}\notin P\ $and
$\overline{z}\notin P.\ $Since $\overline{x}\overline{y}\overline{z}\in
P,\ $we have $4|xyz,\ 4\nmid xy$ and $4\nmid z.\ $Without loss of generality,
we may assume that $2\nmid x.\ $Since $4|xyz,\ $we have $2|y\ $and
$2|z.\ $Since $\overline{x}$ is not unit in $A\ $and $2\nmid x,\ $we have
$3|x.\ $Then we conclude that $\overline{x}\overline{y}\overline{z}%
=\overline{0}$ which is a contradiction. Therefore, $P\ $is a weakly
1-absorbing prime ideal of $A.\ $
\end{example}

\begin{example}
Every weakly 1-absorbing prime ideal is also a weakly 2-absorbing ideal. To
see this, let $P\ $be a weakly 1-absorbing prime ideal of a ring $A.\ $Choose
$x,y,z\ $in $A\ $such that $0\neq xyz\in P.\ $If at least one of the
$x,y,z\ $is unit, then we are done. So assume that $x,y,z$ are nonunits in
$A.\ $Since $P\ $is a weakly 1-absorbing prime ideal, we have $xy\in P\ $or
$z\in P.\ $Thus we have either $xy\in P\ $or $xz\in P\ $or $yz\in P.\ $
\end{example}

\begin{example}
\textbf{(Weakly 2-absorbing ideal that is not weakly 1-absorbing prime)}
Consider the ring $A=%
\mathbb{Z}
_{30}\ $and the ideal $P=(\overline{6}).\ $It is clear that $P\ $is a weakly
2-absorbing ideal of $A.\ $Since $\overline{0}\neq\overline{2}.\overline
{2}.\overline{3}\in P,\ \overline{2}.\overline{2}\notin P\ $and $\overline
{3}\notin P,\ $it follows that $P\ $is not a weakly 1-absorbing prime ideal of
$A.$
\end{example}

\begin{example}
\label{ex1abs}Every 1-absorbing prime ideal is also a weakly 1-absorbing prime ideal.
\end{example}

\begin{example}
\label{ex6}\textbf{(Weakly 1-absorbing prime ideal that is not 1-absorbing
prime)} Let $R=%
\mathbb{Z}
_{pq},\ $where $p\neq q\ $are prime numbers, and let $P=(\overline{0}).\ $Then
$P\ $is clearly a weakly 1-absorbing prime ideal of $A.\ $Since $\overline
{p}.\overline{p}.\overline{q}=\overline{0},\ \overline{p}^{2}\neq\overline
{0}\ $and $\overline{q}\neq\overline{0},\ $it follows that $P\ $is not a
1-absorbing prime ideal of $A.$
\end{example}

By the above examples, we have the following diagram which clarifies the place
of weakly 1-absorbing prime ideals in $L(A)$, which is the lattice of all ideals of $A.$. Here, the arrows in the diagram are irreversible.
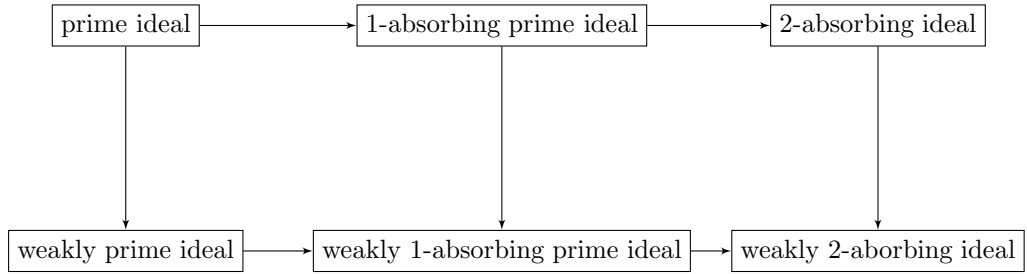
\newline\newline

\begin{figure}[h]
\begin{tikzpicture}
	\tikzset{vertex/.style = {shape=rectangle,draw,minimum size=1.5em}}
	\tikzset{edge/.style = {->,> = latex'}}
	\node[vertex] (a) at  (0,3) {weakly prime ideal};
	\node[vertex] (b) at  (5,3) {weakly 1-absorbing prime ideal};
	\node[vertex] (c) at  (10,3) {weakly 2-aborbing ideal};
	\node[vertex] (d) at  (0,6) {prime ideal};
	\node[vertex] (e) at  (5,6) {1-absorbing prime ideal};
	\node[vertex] (f) at  (10,6) {2-absorbing ideal};
	
	\draw[edge] (d) to (e);
	\draw[edge] (d) to (a);
	\draw[edge] (e) to (f);
	\draw[edge] (e) to (b);
	\draw[edge] (f) to (c);
	\draw[edge] (a) to (b);
	\draw[edge] (b) to (c);
	
	\end{tikzpicture}
\end{figure}

\vspace{0.5cm}

\begin{theorem}
\label{tred}Suppose that $P\ $is a weakly 1-absorbing prime ideal of a reduced
ring $A.\ $Then $\sqrt{P}\ $is a weakly prime ideal of $A.\ $In particular,
$(P:x)\ $is a weakly prime ideal for each $x\in reg(A)-(P\cup u(A)).\ $
\end{theorem}

\begin{proof}
Suppose that $P\ $is a weakly 1-absorbing prime ideal and $A\ $is a reduced
ring. Let $0\neq xy\in\sqrt{P}\ $for some $x,y\in A.\ $If $x\ $is unit, then
clearly we have $y=x^{-1}(xy)\in\sqrt{P}.\ $So assume that $x,y\ $are nonunits
in $A.\ $Since $xy\in\sqrt{P},\ $there exists $n\in%
\mathbb{N}
$ such that $x^{n}y^{n}\in P.\ $This implies that $x^{n}x^{n}y^{n}\in P.\ $As
$A\ $is a reduced ring and $xy\neq0,\ $we have $x^{n}x^{n}y^{n}\neq0.\ $Since
$P\ $is a weakly 1-absorbing prime and $0\neq x^{n}x^{n}y^{n}\in P,\ $we
obtain either $x^{n}x^{n}=x^{2n}\in P$ or $y^{n}\in P.\ $Therefore, we have
$x\in\sqrt{P}\ $or $y\in\sqrt{P}\ $so that $\sqrt{P}\ $is a weakly prime ideal
of $A.\ $On the other hand, choose an element $x\in reg(A)-(P\cup u(A)).\ $We will show that
$(P:x)\ $is a weakly prime ideal. Let $0\neq yz\in(P:x)\ $for some $y,z\in
A.\ $Here, we may assume that $y$ and $z$ are nonunits in $A.\ $Since $x\in
reg(A)\ $and $0\neq yz\in(P:x),\ $we have $0\neq xyz\in P.\ $Since $P\ $is
weakly 1-absorbing prime ideal, we have either $xy\in P\ $or $z\in P.\ $This
yields that $y\in(P:x)\ $or $z\in(P:x).\ $Therefore, $(P:x)\ $is a weakly
prime ideal of $A.$
\end{proof}

\begin{theorem}
\label{thom}Let $A_{1},A_{2}\ $be two commutative rings and $f:A_{1}%
\rightarrow A_{2}\ $be a ring homomorphism such that $f(1_{A_{1}})=1_{A_{2}}.\ $The following
assertions hold.

(i) If $f$ is a monomorphism, $P\ $is a weakly 1-absorbing prime ideal of
$A_{2}\ $and $f(x)\ $is nonunit in $A_{2}\ $for each nonunit $x\in A_{1}%
,\ $then $f^{-1}(P)\ $is a weakly 1-absorbing prime ideal of $A_{1}.$

(ii)\ If $f$ is an epimorphism and $P^{\star}\ $is a weakly 1-absorbing prime
ideal of $A_{1}\ $such that $Ker(f)\subseteq P^{\star},\ $then $f(P^{\star
})\ $is a weakly 1-absorbing prime ideal of $A_{2}.\ $
\end{theorem}

\begin{proof}
(i):\ Let $0\neq xyz\in f^{-1}(P)\ $for some nonunits $x,y,z\in A_{1}.$\ Then
by the assumption, $f(xyz)=f(x)f(y)f(z)\in P$ for some nonunits $f(x),f(y),f(z)\ $%
in $A_{2}.\ $Since $f$ is monomorphism, we have $f(xyz)\neq0.\ $As $P\ $is a
weakly 1-absorbing prime ideal of $A_{2},\ $we conclude either
$f(x)f(y)=f(xy)\in P$ or $f(z)\in P\ $and this implies that $xy\in
f^{-1}(P)\ $or $z\in f^{-1}(P).\ $Hence, $f^{-1}(P)\ $is a weakly 1-absorbing
prime ideal of $A_{1}.$

(ii):\ Suppose that $0\neq x^{\prime}y^{\prime}z^{\prime}\in f(P^{\star})$ for
some nonunits $x^{\prime},y^{\prime},z^{\prime}\in A_{2}.\ $Since $f$ is
epimorphism, there exist nonunits $x,y,z\in A_{1}\ $such that $x^{\prime
}=f(x),y^{\prime}=f(y)\ $and $z^{\prime}=f(z).\ $Then we have $0\neq
f(x)f(y)f(z)=f(xyz)\in f(P^{\star}).\ $As $Ker(f)\subseteq P^{\star},\ $we get
$0\neq xyz\in P^{\star}.\ $Since $P^{\star}\ $is a weakly 1-absorbing prime
ideal of $A_{1},\ $we conclude either $xy\in P^{\star}$ or $z\in P^{\star}%
\ $and this yields that $f(xy)=x^{\prime}y^{\prime}\in f(P^{\star})\ $or
$f(z)=z^{\prime}\in f(P^{\star}).\ $Therefore, $f(P^{\star})\ $is a weakly
1-absorbing prime ideal of $A_{2}.$
\end{proof}

The following example shows that the condition \q{$f(x)$ is nonunit in $A_{2}$ for each nonunit $x\in A_{1}$} is necessary in Theorem \ref{thom}.

\begin{example}
\label{exco1}Let $p$ be a prime number and consider $A_{1}=%
\mathbb{Z}
$ and $A_{2}=%
\mathbb{Z}
_{(p)}.$ Let $f:A_{1}\rightarrow A_{2}\ $be a homomorphism defined by
$f(m)=\frac{m}{1}\ $for each $m\in A_{1}.\ $Then note that $f$ is monomorphism
and also for each prime number $q\neq p,\ $we have $f(q)=\frac{q}{1}\ $is unit
in $A_{2}\ $but $q\ $is nonunit in $A_{1}.\ $Also it is clear that $\frac
{x}{s}\in A_{2}\ $is nonunit if and only if $p|x.\ $Now, put $P=p^{2}%
\mathbb{Z}
_{(p)}.\ $Let $0\neq\frac{x}{s}\frac{y}{t}\frac{z}{u}\in P\ $for some nonunits
$\frac{x}{s},\frac{y}{t},\frac{z}{u}\in A_{2}.\ $Then $p|x,p|y\ $and
$p|z.\ $Then we get $\frac{x}{s}\frac{y}{t}\in p^{2}%
\mathbb{Z}
_{(p)}=P\ $and so $P\ $is a weakly 1-absorbing prime ideal of $A_{2}.\ $On the
other hand, note that $f^{-1}(P)=p^{2}%
\mathbb{Z}
$.\ Since $0\neq pqp=p^{2}q\in f^{-1}(P)\ $and $pq,p\notin f^{-1}(P),\ $we
have $f^{-1}(P)\ $is not weakly 1-absorbing prime ideal of $A_{1}.$
\end{example}

\begin{theorem}
\label{tfac}(i)\ Let $Q\subseteq P\ $be two ideals of $A.\ $If $P\ $is a
weakly 1-absorbing prime ideal of $A,\ $then $P/Q\ $is a weakly 1-absorbing
prime ideal of $A/Q.$

(ii)\ Let $Q\subseteq P\ $be two proper ideals of $A$ such that
$u(A/Q)=\{x+Q:x\in u(A)\}.\ $If $Q\ $is a weakly 1-absorbing prime ideal of
$A\ $and $P/Q\ $is a weakly 1-absorbing prime ideal of $A/Q,\ $then $P$ is a
weakly 1-absorbing prime ideal of $A.$

(iii) If zero ideal is 1-absorbing prime ideal of $A\ $and $P$ is a weakly
1-absorbing prime ideal of $A,\ $then $P\ $is a 1-absorbing prime ideal of A.
\end{theorem}

\begin{proof}
(i):\ Suppose that $P\ $is a weakly 1-absorbing prime ideal of $A.\ $Let
$\overline{0}\neq\overline{x}\overline{y}\overline{z}\in P/Q$ for some
nonunits $\overline{x},\overline{y},\overline{z}\ $in $A/Q\ $where
$\overline{x}=x+Q,\overline{y}=y+Q\ $and $\overline{z}=z+Q\ $for some
$x,y,z\in A.\ $This implies that $0\neq xyz\in P.\ $Since $x,y,z$ are nonunits
in $A\ $and $P\ $is a weakly 1-absorbing prime ideal of $A,\ $we conclude
either $xy\in P\ $or $z\in P.\ $Then we have $\overline{x}\overline{y}\in
P/Q\ $or $\overline{z}\in P/Q.\ $

(ii):\ Let $0\neq xyz\in P\ $for some nonunits $x,y,z$ in $A.\ $If $xyz\in
Q,\ $then either $xy\in Q\subseteq P\ $or $z\in Q\subseteq P,\ $because
$Q\ $is a weakly 1-absorbing prime ideal of $A.\ $So assume that $xyz\notin
Q\ $and thus $0_{A/Q}\neq\overline{x}\overline{y}\overline{z}\in P/Q$, where
$\overline{x}=x+Q,\overline{y}=y+Q\ $and $\overline{z}=z+Q.\ $Also by the
assumption, $\overline{x},\overline{y},\overline{z}\in A/Q\ $are
nonunits.\ Since $P/Q\ $is a weakly 1-absorbing prime ideal of $A/Q,\ $we get
either $\overline{x}\overline{y}=xy+Q\in P/Q\ $or $\overline{z}=z+Q\in P/Q,$
and this gives $xy\in P$ or $z\in P\ $which completes the proof.

(iii):\ Let $xyz\in P\ $for some nonunits $x,y,z\in A.\ $If $xyz\neq0,\ $then
we have either $xy\in P\ $or $z\in P$ because $P\ $is a weakly 1-absorbing
prime ideal of $A.\ $So assume that $xyz=0.\ $Since zero ideal is 1-absorbing
prime ideal, we conclude either $xy=0\in P$ or $z=0\in P$.$\ $Therefore,
$P\ $is a 1-absorbing prime ideal of $A.\ $
\end{proof}

Let $P$ be an ideal of $A.\ $Then $Z_{P}(A)\ $is the set of all elements $x\in
A\ $such that $xy\in P\ $for some $y\notin P,\ $that is, $Z_{P}(A)=\{x\in
A:xy\in P\ $for some $y\notin P\}.$

\begin{theorem}
\label{tloc}Let $A\ $be a ring and $S\subseteq A\ $a multiplicatively closed
set of $A.$

(i)\ If $P\ $is a weakly 1-absorbing prime ideal of $A\ $such that $P\cap
S=\emptyset,\ $then $S^{-1}P\ $is a weakly 1-absorbing prime ideal of
$S^{-1}A.\ $

(ii) If $S\subseteq reg(A)$,\ $u(S^{-1}A)=\{\frac{x}{s}:x\in u(A),s\in
S\}\ $and $S^{-1}P\ $is a weakly 1-absorbing prime ideal of $S^{-1}A$ with
$Z_{P}(A)\cap S =\emptyset,\ $then $P\ $is a weakly 1-absorbing prime ideal of
$A.\ $
\end{theorem}

\begin{proof}
(i): Suppose that $0\neq\frac{x}{s}\frac{y}{t}\frac{z}{u}\in S^{-1}P$ for some
nonunits $\frac{x}{s},\frac{y}{t},\frac{z}{u}\in S^{-1}A.$ Then $0\neq
a(xyz)=(ax)yz\in P\ $for some $a\in S.\ $Also, note that $ax,y,z$ are nonunits
in $A.\ $Since $P\ $is a weakly 1-absorbing prime ideal of $A,\ $we have
either $axy\in P\ $or $z\in P.\ $This implies that $\frac{x}{s}\frac{y}%
{t}=\frac{axy}{ast}\in S^{-1}P\ $or $\frac{z}{u}\in S^{-1}P.\ $Therefore,
$S^{-1}P\ $is a weakly 1-absorbing prime ideal of $S^{-1}A.\ $

(ii):\ Let $0\neq xyz\in P\ $for some nonunits $x,y,z\in A.\ $Since
$S\subseteq reg(A),\ $we conclude that $0\neq\frac{x}{1}\frac{y}{1}\frac{z}%
{1}\in S^{-1}P.\ $Also by the assumption, $\frac{x}{1},\frac{y}{1},\frac{z}{1}$
are nonunits in $S^{-1}A.\ $As $S^{-1}P\ $is a weakly 1-absorbing prime ideal
of $S^{-1}A,\ $we conclude either $\frac{x}{1}\frac{y}{1}=\frac{xy}{1}\in
S^{-1}P$ or $\frac{z}{1}\in S^{-1}P.\ $Then there exists $s\in S\ $such that
$sxy\in P$ or $sz\in P.\ $We may assume that $sxy\in P.\ $If $xy\notin
P,\ $then we have $s\in Z_{P}(A)\cap S$ which is a contradiction. Thus we have
$xy\in P.\ $In other case, similarly, we get $z\in P.\ $Therefore, $P\ $is a weakly
1-absorbing prime ideal of $A.$
\end{proof}

The converse of Theorem \ref{tloc} (i) may not be true unless conditions of the part (ii) are satisfied. See the
following example.

\begin{example}
\label{exco2}Let $p,q$ be distinct prime numbers and $A=%
\mathbb{Z}
.\ $Take the multiplicatively closed set $S=%
\mathbb{Z}
-p%
\mathbb{Z}
$ and $P=p^{2}%
\mathbb{Z}
.$ Then note that $S^{-1}A=%
\mathbb{Z}
_{(p)}$, $S\subseteq reg(A)$ and $ \frac{q}{1}\ $is unit in $S^{-1}A\ $while $q$ is nonunit in
$A$. Moreover, $Z_{P}(A)=p%
\mathbb{Z}
$ and thus $Z_{P}(A)\cap S=\emptyset.\ $ By Example \ref{exco1}, $S^{-1}P=p^{2}%
\mathbb{Z}
_{(p)}$ is a weakly 1-absorbing prime ideal of $S^{-1}A$. However, $P$ is not a weakly 1-absorbing prime ideal of $A$.
\end{example}

We remind that a ring $A\ $is said to be a quasi-local if it has a unique
maximal ideal \cite{Sharp}. Otherwise, we say that $A\ $is non-quasi-local ring.

\begin{theorem}
Let $P$ be a proper ideal of a non-quasi-local ring $A\ $such that
$ann(x)\ $is not a maximal ideal for each $x\in P.\ $The following statements
are equivalent.

(i) $P\ $is a weakly prime ideal of $A.$

(ii)\ $P\ $is a weakly 1-absorbing prime ideal of $A.$
\end{theorem}

\begin{proof}
$(i)\Rightarrow(ii):\ $Follows from Example \ref{exw}.

$(ii)\Rightarrow(i):\ $Let $0\neq xy\in P\ $for some $x,y\in A.\ $If $x$ or
$y$ is unit, then we have either $x\in P$ or $y\in P.\ $So assume that
$x$ and $y$ are nonunits in $A.\ $Since $xy\neq0,\ ann(xy)\ $is a proper ideal.
Then there exists a maximal ideal $\mathfrak{m}_{1}$ of $A\ $such that
$ann(xy)\subsetneq\mathfrak{m}_{1}.\ $As\ $A\ $is non-quasi-local, there
exists a maximal ideal $\mathfrak{m}_{2}\ $of $A.\ $Now, choose $z\in
\mathfrak{m}_{2}-\mathfrak{m}_{1}.\ $Then we have $z\notin ann(xy)\ $and so
$0\neq xyz=zxy\in P.\ $As $P$ is a weakly 1-absorbing prime ideal, we conclude
$zx\in P$ or $y\in P.$ If $y\in P,\ $then we are done. So assume that $zx\in
P.\ $Since $z\notin\mathfrak{m}_{1},\ $there exists $r\in A\ $such that
$1+rz\in\mathfrak{m}_{1}$ and so $1+rz$ is nonunit. Assume that $1+rz\notin
ann(xy).\ $Then we have $0\neq(1+rz)xy\in P.\ $Since $P$ is a weakly
1-absorbing prime ideal, \ we get $(1+rz)x=x+rzx\in P\ $and so $x\in P.\ $Now
assume that $1+rz\in ann(xy).\ $Choose $t\in\mathfrak{m}_{1}%
-ann(xy).\ $Then we have $1+rz+t\in\mathfrak{m}_{1}$ and so $1+rz+t\ $is
nonunit in $A.\ $On the other hand, since $0\neq txy\in P\ $and $P$ is a
weakly 1-absorbing prime ideal,\ we get $tx\in P.\ $As $P$ is a weakly
1-absorbing prime ideal and $0\neq(1+rz+t)xy\in P,\ $we obtain that
$(1+rz+t)x=x+rzx+tx\in P\ $and so $x\in P\ $which completes the proof.
\end{proof}

\begin{theorem}
\label{tmm}Let $A\ $be a commutative ring and $P\ $a proper ideal of $A.\ $The
following statements are equivalent.

(i) $P\ $is a weakly 1-absorbing prime ideal of $A.\ $

(ii)\ For each nonunits $x,y\in A$ with $xy\notin P,\ (P:xy)=P\cup(0:xy).$

(iii)\ For each nonunits $x,y\in A$ with $xy\notin P,\ $either $(P:xy)=P\ $or
$(P:xy)=(0:xy).$

(iv) For each nonunits $x,y\in A\ $and proper ideal $J$ of $A$\ such that
$0\neq xyJ\subseteq P,\ $either $xy\in P\ $or $J\subseteq P.$

(v) For each nonunit $x\in A\ $and proper ideals $I,J\ $of $A\ $such that
$0\neq xIJ\subseteq P,\ $either $xI\subseteq P\ $or $J\subseteq P.$

(vi) For each proper ideals $I,J,K\ $of $A\ $such that $0\neq IJK\subseteq
P,\ $either $IJ\subseteq P\ $or $K\subseteq P.\ $
\end{theorem}

\begin{proof}
$(i)\Rightarrow(ii):$\ Suppose that $P\ $is a weakly 1-absorbing prime ideal
of $A.\ $Now, choose nonunits $x,y\in A$ with $xy\notin P.\ $Let
$z\in(P:xy).\ $If $zxy=0,\ $then we are done. So assume that $xyz\neq
0.\ $Since $xy\notin P\ $and $xyz\in P,\ $we have $z$ is not unit. As $P\ $is
weakly 1-absorbing prime, we conclude $z\in P.\ $Thus we have
$\ (P:xy)\subseteq P\cup(0:xy).$\ Since the reverse inclusion always hold, we
get $(P:xy)=P\cup(0:xy).$

$(ii)\Rightarrow(iii):$ Follows from the fact that if an ideal is a union of
two ideals, then it must be one of them.

$(iii)\Rightarrow(iv):\ $Suppose that $0\neq xyJ\subseteq P$ for some nonunits
$x,y\in A\ $and a proper ideal $J\ $of $A.\ $Assume that $xy\notin P.\ $Then
by (iii), we have either $(P:xy)=(0:xy)\ $or $(P:xy)=P.\ $Suppose
$(P:xy)=(0:xy).\ $Since $xyJ\subseteq P,\ $we get $J\subseteq(P:xy)=(0:xy)\ $%
and so $xyJ=0$ which is a contradiction. Thus $(P:xy)=P$ and this implies that
$J\subseteq(P:xy)=P\ $which completes the proof.

$(iv)\Rightarrow(v):\ $Let $0\neq xIJ\subseteq P$ for some nonunit $x\in
A\ $and proper ideals $I,J\ $of $A.\ $Assume that $xI\nsubseteq P\ $and
$J\nsubseteq P.\ $As $xI\nsubseteq P,\ $there exists $a\in I\ $such that
$xa\notin P.\ $Since $xIJ\neq0,\ $there exists $a_{0}\in I\ $such that
$xa_{0}J\neq0.\ $Now, we will show that $xaJ=0.\ $Otherwise, by (iv), we would
have $xa\in P\ $or $J\subseteq P\ $since $a$ is not unit,\ which is
contradiction. Thus $xaJ=0$ and so we have $0\neq xa_{0}J=x(a+a_{0}%
)J\subseteq P.\ $Then by\ (iv), we have $x(a+a_{0})\in P$ since $a+a_0$ is not unit. On the other hand,
since $0\neq xa_{0}J\subseteq P,\ $we have $xa_{0}\in P.\ $As $x(a+a_{0})\in
P\ $and $xa_{0}\in P,\ $we conclude that $xa\in P\ $which is a contradiction.
Therefore, $xI\subseteq P\ $or $J\subseteq P.$

$(v)\Rightarrow(vi):\ $Suppose that $0\neq IJK\subseteq P\ $for some proper
ideals $I,J,K\ $of $A.\ $Assume that $IJ\nsubseteq P$ and $K\nsubseteq
P.\ $Then there exists $y\in I\ $such that $yJ\nsubseteq P.\ $If $yJK\neq
0,\ $then we have either $yJ\subseteq P\ $or $K\subseteq P\ $which is
contradiction. Thus $yJK=0.\ $Since $IJK\neq0,$ there exists $x\in I\ $such
that $xJK\neq0.\ $As $xJK\subseteq P,\ $we conclude that $xJ\subseteq P.\ $As
$0\neq xJK=(x+y)JK\subseteq P,\ $we have $(x+y)J\subseteq P\ $and this yields
that $yJ\subseteq P\ $which is a contradiction.

$(vi)\Rightarrow(i):\ $Suppose that $0\neq xyz\in P\ $for some nonunits
$x,y,z\in A.\ $Now, put $I=xA,J=yA\ $and $K=zA.\ $Then $0\neq
IJK=xyzA\subseteq P\ $for proper ideals $I,J$ and $K$ of $A$. Then by (vi), we
have $xy\in IJ\subseteq P\ $or $z\in K\subseteq P\ $, as needed.
\end{proof}

\begin{definition}
Let $P\ $be a weakly 1-absorbing prime ideal of $A\ $and $x,y,z$ be nonunits
in $A.\ (x,y,z)\ $is said to be a \textit{1-triple zero} of $P\ $provided that
$xyz=0,\ xy\notin P$ and $z\notin P.\ $
\end{definition}

Note that a weakly 1-absorbing prime ideal $P\ $of $A\ $is not a 1-absorbing
prime ideal if and only if $P\ $has a 1-triple zero.

\begin{theorem}
\label{ttriple}Let $P\ $be a weakly 1-absorbing prime ideal of $A\ $and
$(x,y,z)$ be a 1-triple zero of $P.\ $Then,

(i)\ $xyP=0.$

(ii) If $x,y\notin(P:z),\ $then $xzP=0=yzP=xP^{2}=yP^{2}=zP^{2}.$ In
particular, $P^{3}=0.\ $
\end{theorem}

\begin{proof}
$(i):\ $Suppose that $P\ $is a weakly 1-absorbing prime ideal of $A\ $and
$(x,y,z)$ is a 1-triple zero of $P.$ Assume that $xyP\neq0.\ $Then there
exists $p\in P\ $such that $0\neq xyp.\ $Since $(x,y,z)$ is a 1-triple zero of
$P,\ xyz=0,\ xy\notin P\ $and $z\notin P.\ $This implies that $0\neq
xyp=xy(z+p)\in P.\ $Since $xy\notin P,\ z+p\ $is not unit. As $P\ $is a weakly
1-absorbing prime ideal of $A,\ $we conclude that $z+p\in P\ $and so $z\in
P\ $which is a contradiction. Therefore, $xyP=0.$

$(ii):\ $Let$\ x,y\notin(P:z).\ $Then $xz,yz\notin P.\ $Now, take $p\in
P.\ $Since $xyz=0,\ $we have $xy(z+p)=xzp\in P.\ $If $z+p\ $is unit, we get
$xy\in P\ $which is a contradiction. Thus $z+p\ $is nonunit. If $xzp\neq
0,\ $we conclude $0\neq xy(z+p)\in P.\ $As $P\ $is a weakly 1-absorbing
prime ideal of $A,\ $we get $z+p\in P$ and so $z\in P\ $which is a
contradiction. Therefore, $xzp=0\ $and this yields $xzP=0.\ $One can
similarly show that $yzP=0.\ $Now, we will show that $xP^{2}=0.\ $Assume that
$xP^{2}\neq0.\ $Then there exists $q,p\in P\ $such that $xpq\neq0.\ $Then we
have $0\neq xpq=x(y+p)(z+q)=xyz+xyq+xzp+xpq\in P\ $since $xyz=0,\ xyP=0\ $and
$xzP=0.\ $If $y+p\ $is unit, then $x(z+q)\in P\ $and so $xz\in P\ $which is a
contradiction. Thus $y+p\ $is nonunit. Likewise, $z+q\ $is nonunit. As $P\ $is
a weakly 1-absorbing prime ideal, we conclude that either $x(y+p)\in P\ $or
$z+q\in P\ $and this yields that $xy\in P\ $or $z\in P\ $which are both
contradictions. Therefore, $xP^{2}=0.\ $Similar argument shows that
$yP^{2}=zP^{2}=0.\ $

Now, we will show that $P^{3}=0.\ $Suppose to the contrary. Then there exists
$p,q,r\in P\ $such that $pqr\neq0.\ $Then we have $(p+x)(q+y)(r+z)=pqr$ since
$zP^{2}=0=yP^{2}=yzP=xP^{2}=xzP=xyP\ $and $xyz=0.\ $This implies that $0\neq
pqr=(p+x)(q+y)(r+z)\in P.\ $If $(p+x)\ $is unit, then $(q+y)(r+z)\in P.\ $As
$q,r\in P,\ $we get $yz\in P\ $which is a contradiction. Thus $(p+x)\ $is
nonunit. Similarly, $(q+y),(r+z)\ $are not units in $A.\ $Since $P\ $is a
weakly 1-absorbing prime ideal of $A,\ $we conclude that $(p+x)(q+y)\in P\ $or
$r+z\in P\ $and thus we have $xy\in P\ $or $z\in P,\ $again a contradiction.
Hence $P^{3}=0.\ $
\end{proof}

\begin{theorem}
(i) Let $P\ $be a weakly 1-absorbing prime ideal of a reduced ring $A\ $that
is not 1-absorbing prime$.\ $Suppose that $(x,y,z)\ $is a 1-triple zero of
$P\ $such that $x,y\notin(P:z).\ $Then $P=0.\ $

(ii)\ Let $P\ $be a nonzero weakly 1-absorbing prime ideal of\ a reduced ring
$A$ that is not 1-absorbing prime ideal$.\ $If $(x,y,z)\ $is a 1-triple zero
of $P,\ $then $xz\in P\ $or $yz\in P.\ $
\end{theorem}

\begin{proof}
(i):\ Let $P\ $be a weakly 1-absorbing prime ideal that is not 1-absorbing
prime ideal. Suppose that $(x,y,z)\ $is a 1-triple zero of $P\ $with
$x,y\notin(P:z).\ $Then by Theorem \ref{ttriple} (iii),\ $P^{3}=0.\ $Since
$A\ $is reduced, we get $P=0.\ $

(ii): Let $(x,y,z)\ $be a 1-triple zero of $P.\ $If $xz\ $and $yz\notin
P,\ $then $x,y\notin(P:z)\ $so by (i), we conclude that $P=0\ $which is a contradiction.
\end{proof}

Let $A\ $be a ring and $M$ be an $A$-module. The \textit{trivial extension}
(or sometimes called Nagata's idealization) $A\ltimes M=A\oplus M$ is a
commutative ring with componentwise addition and the multiplication given by
$(x,m)(y,m^{\prime})=(xy,xm^{\prime}+ym)\ $for each $x,y\in A;m,m^{\prime}\in
M\ $\cite{Na},\ \cite{AnWi}.

For any $A$-module $M,\ $the set of annihilators of $M$ is denoted by
$ann(M)=\{x\in A:xM=(0)\}$. For an ideal $P$ of $A$ and a submodule $N$ of $M$  The set $P\ltimes N\ $is not always an ideal of
$A\ltimes M\ $and it is an ideal in trivial extension if and only if
$PM\subseteq N$ \cite[Theorem 3.1]{AnWi}$.\ $The authors showed in \cite{AnWi}
that every prime and maximal ideal $J$ of $A\ltimes M\ $has the form
$J=P\ltimes M$,\ where $P\ $is prime and maximal ideal, respectively. Also,
Anderson and Smith in \cite{AnSmi} determine when $P\ltimes M$ is a weakly
prime ideal in $A\ltimes M.\ $Now, we will give a similar result for weakly
1-absorbing prime ideals.

\begin{theorem}
\label{ttri}Let $P$ be a proper ideal of a ring $A$ and $M\ $ be an $A$%
-module$.\ $The following statements are equivalent.

(i)\ $P\ltimes M\ $is a weakly 1-absorbing prime ideal of $A\ltimes M.$

(ii)\ $P\ $is a weakly 1-absorbing prime ideal of $A\ $and if $xyz=0\ $for
some nonunits $x,y,z\in A\ $with $xy\notin P\ $and $z\notin P,\ $then
$xy,xz,yz\in ann(M).$
\end{theorem}

\begin{proof}
Firstly, we will show that
$(x,m)\ $is nonunit in $A\ltimes M$ if and only if $x$ is nonunit in $A.\ $ Let $(x,m)$ be a unit of $A\ltimes M.\ $. Then there exists
$(y,m^{\prime})\in A\ltimes M$ such that $(x,m)(y,m^{\prime})=(1,0).\ $This
implies that $xy=1\ $and $xm^{\prime}+ym=0\ $, and so $x$ is a unit of $A.\ $For
the converse, assume that $x$ is a unit of $A.\ $Now, we will show that
$(x,m)$ is unit for all $m\in M.\ $First, choose $y\in A\ $such that
$xy=1.\ $Now, put $m^{\prime}=-y^{2}m.\ $Then note that $(x,m)(y,m^{\prime
})=(xy,xm^{\prime}+ym)=(1,0)$ and so $(x,m)\ $is unit.

$(i)\Rightarrow(ii):\ $Suppose that $P\ltimes M\ $is a weakly 1-absorbing
prime ideal of $A\ltimes M.\ $Let $0\neq xyz\in P\ $for some nonunits
$x,y,z\in A.\ $Then we have $(x,0)(y,0)(z,0)=(xyz,0)\in P\ltimes M\ $for some
nonunits $(x,0),(y,0),(z,0)\in A\ltimes M.\ $Since $P\ltimes M\ $is weakly
1-absorbing prime,\ we have either $(x,0)(y,0)=(xy,0)\in P\ltimes M$ or
$(z,0)\in P\ltimes M.\ $Then we obtain $xy\in P$ or $z\in P.\ $Thus $P\ $is a
weakly 1-absorbing prime ideal of $A.\ $Now, assume that $xyz=0\ $for some
nonunits $x,y,z\in A\ $with $xy\notin P\ $and $z\notin P.\ $Assume that
$xy\notin ann(M).\ $Then there exists $m\in M\ $such that $xym\neq0.\ $This
implies that $(x,0)(y,0)(z,m)=(xyz,xym)\neq(0,0).\ $As
$(x,0)(y,0)(z,m)=(xyz,abm)\in P\ltimes M\ $for some nonunits
$(x,0),(y,0),(z,m)\in P\ltimes M,\ $we conclude either $(x,0)(y,0)=(xy,0)\in
P\ltimes M\ $or $(z,m)\in P\ltimes M\ $and so we have $xy\in P\ $or $z\in
P,\ $a contradiction. Thus $xy\in ann(M).\ $Similarly, we have $xz,yz\in
ann(M).\ $

$(ii)\Rightarrow(i):\ $Let $(0,0)\neq(x,m_{1})(y,m_{2})(z,m_{3})=(xyz,xym_{3}%
+xzm_{2}+yzm_{1})\in P\ltimes M\ $for some nonunits $(x,m_{1}),(y,m_{2}%
),(z,m_{3})\in A\ltimes M.\ $Then we have $xyz\in P\ $for some nonunits
$x,y,z\in A.\ $If $xyz\neq0,\ $then we conclude either $xy\in P\ $or $z\in
P.\ $This implies that $(x,m_{1})(y,m_{2})=(xy,xm_{2}+ym_{1})\in P\ltimes M$
or $(z,m_{3})\in P\ltimes M.\ $Now, assume that $xyz=0.$\ If $xy\notin P\ $and
$z\notin P,\ $then by assumption, we have $xy,xz,yz\in ann(M)\ $and so
$xym_{3}+xzm_{2}+yzm_{1}=0.\ $Then we have $(x,m_{1})(y,m_{2})(z,m_{3}%
)=(0,0)\ $which is a contradiction. Thus we have either $xy\in P\ $or $z\in P$
and so either$\ (x,m_{1})(y,m_{2})\in P\ltimes M\ $or $(z,m_{3})\in P\ltimes
M.\ $
\end{proof}

\begin{theorem}
\label{tcar}Suppose that $A_{1},A_{2}\ $be two commutative rings that are not
fields and $A=A_{1}\times A_{2}.\ $Let $P\ $be a nonzero proper ideal of
$A.\ $The following assertions are equivalent.

(i)\ $P\ $is a weakly 1-absorbing prime ideal of $A.\ $

(ii)\ $P=P_{1}\times A_{2}\ $for some prime ideal $P_{1}\ $of $A_{1}\ $or
$P=A_{1}\times P_{2}\ $for some prime ideal $P_{2}\ $of $A_{2}.$

(iii)\ $P$ is a prime ideal of $A.$

(iv)\ $P\ $is a weakly prime ideal of $A.\ $

(v)\ $P\ $is a 1-absorbing prime ideal of $A.$
\end{theorem}

\begin{proof}
$(i)\Rightarrow(ii):\ $Let $P\ $be a nonzero proper ideal of $A.\ $Then we can
write $P=P_{1}\times P_{2}\ $for some ideals $P_{1}\ $of $A_{1}\ $and
$P_{2}\ $of $A_{2}.$\ Since $P$ is nonzero, $P_{1}\neq0$ or $P_{2}\neq
0.\ $Without loss of generality, we may assume that $P_{1}\neq0.\ $Then there
exists $0\neq x\in P_{1}.\ $Since $P\ $is a weakly 1-absorbing prime ideal and
$(0,0)\neq(1,0)(1,0)(x,1)\in P,\ $we conclude either $(1,0)\in P\ $or
$(x,1)\in P.\ $Then we have either $P_{1}=A_{1}\ $or $P_{2}=A_{2}.\ $Assume
that $P_{1}=A_{1}.\ $Now we will show that $P_{2}\ $is a prime ideal of
$A_{2}.\ $Let $yz\in P_{2}\ $for some $y,z\in A_{2}.\ $If $y$ or $z$ is a
unit, then we have either $y\in P_{2}\ $or $z\in P_{2}.\ $So assume that
$y,z\ $are nonunits in $A_{2}.\ $Since $A_{1}\ $is not a field, there exists a
nonzero nonunit $t\in A_{1}.\ $This implies that $(0,0)\neq
(t,1)(1,y)(1,z)=(t,yz)\in P.\ $As $P\ $is a weakly 1-absorbing prime ideal of
$A,\ $we conclude either $(t,1)(1,y)=(t,y)\in P\ $or $(1,z)\in P.\ $Thus we
get $y\in P_{2}\ $or $z\in P_{2}\ $and so $P_{2}\ $is a prime ideal of
$A_{2}.$\ In other case, one can similarly show that $P=P_{1}\times A_{2}$ and
$P_{1}\ $is a prime ideal of $A_{1}.$

$(ii)\Rightarrow(iii):\ $It is clear.

$(iii)\Leftrightarrow(iv):\ $Follows from \cite[Theorem 7]{AnSmi}.

$(iii)\Rightarrow(v):\ $Directly from definition \cite[Definition 2.1]{YasNik}.

$(v)\Rightarrow(i):\ $Follows from Example \ref{ex1abs}.
\end{proof}

\begin{theorem}
\label{tcarr}Let $A_{1},A_{2},\ldots,A_{n}\ $be commutative rings and
$A=A_{1}\times A_{2}\times\cdots\times A_{n},\ $where $n\geq2.\ $The following
statments are equivalent.

(i)\ Every proper ideal of $A\ $is a weakly 1-absorbing prime ideal.

(ii)\ $n=2\ $and $A_{1},A_{2}\ $are fields.
\end{theorem}

\begin{proof}
$(i)\Rightarrow(ii):\ $Suppose that $n\geq3.\ $Let $P=\{0\}\times\{0\}\times
A_{3}\times A_{4}\times\cdots\times A_{n}.\ $Then by (i),\ $P\ $is a weakly
1-absorbing prime ideal of $A.\ $Choose a nonzero element $x\in A_{3}.\ $Then
we have $(1,0,1,1,\ldots,1)(1,0,1,1,\ldots,1)(0,1,x,1,1,\ldots
,1)=(0,0,x,1,1,\ldots,1)\in P.\ $As $P\ $is a weakly 1-absorbing prime ideal,
we conclude either $(1,0,1,1,\ldots,1)(1,0,1,1,\ldots,1)=(1,0,1,1,\ldots,1)\in
P$ or $(0,1,x,1,1,\ldots,1)\in P\ $which both of them are contradictions.
Therefore $n=2.\ $Now, we will show that $A_{1},A_{2}\ $are fields. Now, put
$P^{\star}=\{0\}\times A_{2}.\ $Since $P^{\star}$ is a weakly 1-absorbing
prime ideal of $A,\ $by Theorem \ref{tcar},\ $\{0\}\ $is a prime ideal of
$A_{1},\ $that is, $A_{1}\ $is a domain. Likewise, $A_{2}\ $is a domain. Let
$0\neq a\in A_{1}.\ $If $a$ is nonunit, then $J=(a^{2})\times A_{2}\ $is a
proper ideal of $A.\ $Then by assumption, $J\ $is a weakly 1-absorbing prime
ideal of $A.\ $Since $(0,0)\neq(a,1)(1,0)(a,1)\in J,\ $we conclude either
$(a,1)(1,0)=(a,0)\in J\ $or $(a,1)\in J\ $and this yields that $a=ra^{2}\ $for
some $r\in A_{1}.\ $As $A_{1}\ $is a domain and $a\neq0,\ $we have
$1=ra$,\ namely, $a$ is a unit\ which is a contradiction. Therefore, $a$ is
unit and $A_{1}\ $is a field. Likewise, $A_{2}\ $is a field.

$(ii)\Rightarrow(i):\ $Suppose that $n=2\ $and $A_{1},A_{2}\ $are fields. Let
$P\ $be a proper ideal of $A=A_{1}\times A_{2}.\ $Then all the possibilities
for $P\ $are $\{0\}\times\{0\},\ \{0\}\times A_{2}\ $and $A_{1}\times
\{0\}.\ $If $P=\{0\}\times A_{2}$ or $P=A_{1}\times\{0\},\ $then $P\ $is weakly 1-absorbing prime ideal by Theorem \ref{tcar}. If $P=\{0\}\times\{0\},\ $then
$P\ $is trivially weakly 1-absorbing prime ideal of $A.$
\end{proof}

\section{Rings in which every proper ideal is weakly 1-absorbing prime}

\begin{proposition}
\label{pc1}Let $A\ $be a ring in which every proper ideal is weakly
1-absorbing prime. Then one of the following statements hold.

(i)\ $Jac(A)^{2}=(0).$

(ii)\ For each $x,y\in Jac(A)\ $with $xy\neq0,\ $then $Jac(A)=(0:xy).\ $In
particular, $Jac(A)=(0:Jac(A)^{2}).\ \ $
\end{proposition}

\begin{proof}
Suppose that $Jac(A)^{2}\neq(0).\ $Choose $x,y\in Jac(A)\ $such that
$xy\neq0.\ $Now, we will show that $Jac(A)\subseteq(0:xy).\ $Suppose to the
contrary. Then there exists $z\in Jac(A)-(0:xy),\ $which implies that
$xyz\neq0.\ $Since every ideal is weakly 1-absorbing prime, so is
$P=Axyz$.\ As $0\neq xyz\in P,\ $we conclude either $xy\in P$ or $z\in
P.\ $Thus we have $xy=axyz$ or $z=bxyz$ for some $a,b\in A.\ $This gives
$xy(1-az)=0$ or $z(1-bxy)=0.\ $Since $x,y,z\in Jac(R),\ $we have $1-bxy$ and
$1-az$ are unit, and so we have either $xy=0$ or $z=0\ $which is
contradiction. Thus we have $Jac(A)\subseteq(0:xy).\ $Now, take
$c\in(0:xy).\ $Let $c\notin Jac(A)$. Then there exists $d\in A\ $such that
$1-cd$ is nonunit. Then note that $(1-cd)xy\neq0$ since $cxy=0\ $and
$xy\neq0.\ $Since $Q=A(1-cd)xy$ is weakly 1-absorbing prime, we get either
$(1-cd)x\in Q\ $or $y\in Q.\ $Then there exist $r,s\in A\ $such that
$(1-cd)x=(1-cd)xyr$ or $y=(1-cd)xys$. Then we obtain $(1-cd)x(1-yr)=0$ or
$y(1-(xs-cdxs))=0.\ $Since $1-yr$ and $1-(xs-cdxs)\ $are units, we get
$(1-cd)x=0\ $or $y=0,\ $which again are contradictions. Thus we have
$(0:xy)\subseteq Jac(A),\ $that is, $Jac(A)=(0:xy).\ $Now,we will show that
$Jac(A)=(0:Jac(A)^{2}).\ $First note that $Jac(A)^{2}=\sum\limits_{x,y\in
Jac(A)}Axy.\ $Then we have $(0:Jac(A)^{2})=\bigcap\limits_{x,y\in
Jac(A)}(0:xy).$ Let $x,y\in Jac(A).\ $If $xy=0,\ $then $(0:xy)=A.\ $If
$xy\neq0,\ $then we have $(0:xy)=Jac(A).\ $So we conclude that $(0:Jac(A)^{2}%
)=\bigcap\limits_{x,y\in Jac(A)}(0:xy)=Jac(A).$
\end{proof}

\begin{theorem}
\label{tql}Let $(A,\mathfrak{m})$ be a quasi-local ring. The following
statements are equivalent.

(i) Every proper ideal is a weakly 1-absorbing prime ideal.

(iii)\ $\mathfrak{m}^{3}=(0).$
\end{theorem}

\begin{proof}
$(i)\Rightarrow(ii):\ $First note that $Jac(A)=\mathfrak{m}$.\ The rest
follows from Proposition \ref{pc1}.

$(ii)\Rightarrow(i):\ $Suppose that $\mathfrak{m}^{3}=(0).$ Let $P\ $be a
nonzero proper ideal of $A.\ $Assume that $P\ $is not weakly 1-absorbing
prime, then there exist nonunits $x,y,z\in A\ $such that $0\neq xyz\in P\ $but
$xy\notin P\ $and $z\notin P.\ $As $x,y,z$ are nonunits and $\mathfrak{m}%
^{3}=(0),\ $we get $xyz=0$ which is a contradiction. Therefore, $P\ $is weakly
1-absorbing prime.
\end{proof}

\begin{corollary}
Suppose that $(A,\mathfrak{m})$ is a quasi-local ring with $\mathfrak{m}%
^{2}=(0).\ $Then, every proper ideal is a 1-absorbing prime ideal.
\end{corollary}

\begin{proof}
Let $(A,\mathfrak{m})$ be a quasi-local ring with $\mathfrak{m}^{2}=(0).\ $Let
$P\ $be a proper ideal of $A\ $and $xyz\in P$ for some nonunits $x,y,z\in
A.\ $Since $x$ and $y$ are nonunit and $\mathfrak{m}^{2}=(0),\ $we have
$xy=0\in P.\ $Therefore, $P\ $is a 1-absorbing prime ideal of $A.\ $
\end{proof}

\begin{theorem}
\label{tmax}Suppose that every proper ideal of $A$ is weakly 1-absorbing
prime. Then $\left\vert Max(A)\right\vert \leq2.$
\end{theorem}

\begin{proof}
Let $A\ $be a ring over which every proper ideal is weakly 1-absorbing prime.
Suppose that $\left\vert Max(A)\right\vert \geq3.\ $Choose maximal ideals
$\mathfrak{m}_{1},\mathfrak{m}_{2},\mathfrak{m}_{3}.\ $

\textbf{Case 1: }Assume that $\mathfrak{m}_{1}\mathfrak{m}_{2}\mathfrak{m}%
_{3}\neq(0).\ $Since $\mathfrak{m}_{1}\mathfrak{m}_{2}\mathfrak{m}%
_{3}\subseteq\mathfrak{m}_{1}\mathfrak{m}_{2}\mathfrak{m}_{3}\ $and
$\mathfrak{m}_{1}\mathfrak{m}_{2}\mathfrak{m}_{3}$ is weakly 1-absorbing prime
ideal, by Theorem \ref{tmm}, we conclude either $\mathfrak{m}_{1}%
\mathfrak{m}_{2}\subseteq\mathfrak{m}_{1}\mathfrak{m}_{2}\mathfrak{m}%
_{3}\subseteq\mathfrak{m}_{3}$ or $\mathfrak{m}_{3}\subseteq\mathfrak{m}%
_{1}\mathfrak{m}_{2}\mathfrak{m}_{3}\subseteq\mathfrak{m}_{1}\mathcal{.\ }%
$This implies that $\mathfrak{m}_{1}=\mathfrak{m}_{3}$ or $\mathfrak{m}%
_{2}=\mathfrak{m}_{3}\ $, a contradiction.

\textbf{Case 2:\ }Assume that $\mathfrak{m}_{1}\mathfrak{m}_{2}\mathfrak{m}%
_{3}=(0).\ $Then by Chinese Remainder Theorem, $A\ $is isomorphic to
$\left(  A/\mathfrak{m}_{1}\right) \times \left(  A/\mathfrak{m}_{2}\right)
\times \left(  A/\mathfrak{m}_{3}\right)  .\ $Take $A=F_{1}\times F_{2}\times
F_{3},\ $where $F_{1},F_{2},F_{3}\ $are fields. Then by Thoerem \ref{tcarr}%
,\ we have $F_{1}=(0)\ $or $F_{2}=(0)\ $or $F_{3}=(0)\ $which is contradiction.

Therefore, $\left\vert Max(A)\right\vert \leq2.$
\end{proof}

Now, we characterize rings whose all proper ideals are weakly 1-absorbing
prime ideal.

\begin{theorem}
\label{tring}Let $A\ $be a ring. The following statements are equivalent.

(i)\ Every proper ideal of $A\ $is a weakly 1-absorbing prime ideal.

(ii) Either $(A,\mathfrak{m})$ is a quasi-local ring with $\mathfrak{m}%
^{3}=(0)$ or $A=F_{1}\times F_{2},\ $where $F_{1},F_{2}\ $are fields.
\end{theorem}

\begin{proof}
$(i)\Rightarrow(ii):\ $Suppose that every proper ideal of $A\ $is a weakly
1-absorbing prime ideal. Then by Theorem \ref{tmax}, $\left\vert
Max(A)\right\vert \leq2.$ Suppose that $\left\vert Max(A)\right\vert
=1,\ $that is, $(A,\mathfrak{m})$ is a quasi-local ring.\ Then by Theorem
\ref{tql}, $\mathfrak{m}^{3}=(0).$ Now, suppose that $Max(A)=\{\mathfrak{m}%
_{1},\mathfrak{m}_{2}\}.\ $Then by Proposition \ref{pc1}, we have
$Jac(A)^{3}=\mathfrak{m}_{1}^{3}\mathfrak{m}_{2}^{3}=(0).\ $By Chinese
Remainder Theorem, we conclude that $R=\left(  R/\mathfrak{m}_{1}^{3}\right)
\times\left(  R/\mathfrak{m}_{2}^{3}\right)  .\ $So by Theorem \ref{tcarr},
$\left(  R/\mathfrak{m}_{1}^{3}\right)  \ $and $\left(  R/\mathfrak{m}_{2}%
^{3}\right)  $ are fields.

$(ii)\Rightarrow(i):\ $Suppose first that $(A,\mathfrak{m})$ is a quasi-local
ring with $\mathfrak{m}^{3}=(0).\ $Then by Theorem \ref{tql}, every proper
ideal is weakly 1-absorbing prime. Now, assume that $A=F_{1}\times F_{2}%
,\ $where $F_{1},F_{2}\ $are fields. Then by Theorem \ref{tcarr}, every proper
ideal is weakly 1-absorbing prime ideal.
\end{proof}

\begin{example}
Let $A=%
\mathbb{Z}
_{n},\ $where $n\geq2\ $is an integer. Then by previous theorem, every ideal
of $A\ $is weakly 1-absorbing prime if and only if $n=p^{3}\ $or $n=p_{1}%
p_{2}\ $for some prime numbers $p$ and $p_{1}\neq p_{2}.$
\end{example}

\section{Weakly 1-absorbing prime $z$-ideals in $C(X)$}

Let $X\ $be a topological space and $C(X)\ $be the ring of all real valued
continuous functions on a topological space $X.\ $Note that $C(X)\ $is a
commutative ring with a nonzero identity $i(x)=1\ $for each $x\in X.\ $It is
known that for each topological space $X,\ C(X)$ is isomorphic to $C(Y)\ $for
some completely regular space $Y$ \cite{GilJe}$.\ $So from now on, we assume
that $X\ $is a completely regular topological space. We refer \cite{GilJe} to
the reader for more details on $C(X).\ $The notion of $z$-filter has an
important role in studying the algebraic properties of $C(X).$ Let $f\in
C(X).\ $The zero set $f$ of $C(X)\ $is denoted by $z(f)=\{x\in X:f(x)=0\}$ and
also the set of all zero sets in $X\ $is denoted by $z(X).\ $Recall
from$\ $\cite{GilJe} that a nonempty subset $\mathcal{F}\subseteq z(X)\ $is
said to be a $z$-filter if the following conditions hold: (i) $\emptyset
\notin$ $\mathcal{F}$(ii)\ $Z_{1}\cap Z_{2}\in\mathcal{F}$ for each
$Z_{1},Z_{2}\in\mathcal{F}$ and (iii)\ $Z_{1}\subseteq Z,\ Z_{1}%
\in\mathcal{F}$ and $Z\in z(X)\ $imply that $Z\in\mathcal{F}$. Let $I\ $be
an ideal of $C(X).\ $Then the set $z[I]=\{z(f):f\in I\}\ $is an example of
$z$-filter. Also, if $\mathcal{F}$ is a $z$-filter, then $z^{-1}%
[\mathcal{F}]=\{f\in C(X):z(f)\in\mathcal{F}\}$ is an ideal of $C(X).\ $ For an ideal $I\ $of $C(X),\ I\subseteq z^{-1}[z[I]]\ $always holds
but the converse is not true in general (See, \cite{GilJe}).\ Recall from
\cite{GilJe} that an ideal $I\ $in $C(X)\ $is said to be a $z$-ideal if
$I=z^{-1}[z[I]]$, or equivalently, $z(f)\in z[I]\ $implies that $f\in I.$ Note
that every $z$-ideal $I\ $of $C(X)\ $is a radical ideal, that is, $I=\sqrt
{I}.\ $

\begin{definition}
(\cite{GilJe})\ A $z$-filter $\mathcal{F}$ is called a prime $z$-filter if
$Z_{1}\cup Z_{2}\in\mathcal{F}$ for some $Z_{1},Z_{2}\in z(X)\ $implies
either $Z_{1}\in\mathcal{F}$ or $Z_{2}\in\mathcal{F}$
\end{definition}

Note that there is a correspondence between the prime $z$-filters and the
prime $z$-ideals in $C(X)$ (See, \cite{GilJe}). Recently, the authors in \cite{BilDemOr},
introduced the concept of 2-absorbing $z$-filters and used them to study
2-absorbing ideals in $C(X).\ $Recall from \cite{BilDemOr} that a $z$-filter
$\mathcal{F}$ is called a 2-absorbing $z$-filter if $Z_{1}\cup Z_{2}\cup
Z_{3}\in\mathcal{F}$ for some $Z_{1},Z_{2},Z_{3}\in z(X)\ $implies either
$Z_{1}\cup Z_{2}\in\mathcal{F}$ or $Z_{2}\cup Z_{3}\in\mathcal{F}$ or
$Z_{1}\cup Z_{3}\mathcal{F}$. Note that every prime $z$-filter is clearly a
2-absorbing $z$--filter but the converse is not true in general. For instance,
in $C(%
\mathbb{R}
),\ \mathcal{F}=\{Z\in z(%
\mathbb{R}
):\{0,\frac{\pi}{2}\}\subseteq Z\}\ $is not a prime $z$-filter because $z(\sin
x)\cup z(\cos x)\in\mathcal{F}$ but $z(\sin x),z(\cos x)\notin\mathcal{F}%
$Also, one can easily see that $\mathcal{F}$ is a 2-absorbing $z$-filter. The
authors in \cite{BilDemOr} showed that if $P\ $is a 2-absorbing ideal of
$C(X),\ $then $z[P]\ $is a 2-absorbing $z$-filter and also if $\mathcal{F}$
is a 2-absorbing $z$-filter, then $z^{-1}[\mathcal{F}]\ $is a 2-absorbing
$z$-ideal of $C(X)\ $(See, \cite[Theorem 7]{BilDemOr}).

In this section, we investigate weakly 1-absorbing prime $z$-ideal, weakly
prime $z$-ideal in $C(X).\ $Darani in his paper \cite{Dar}, gave a
generalization of prime $z$-filters to study weakly prime ideals in
$C(X).\ $Recall from \cite{Dar} that a $z$-filter $\mathcal{F}$ is said to be
a weakly prime $z$-filter if $X\neq Z_{1}\cup Z_{2}\in\mathcal{F}$ for some
$Z_{1},Z_{2}\in z(X)\ $implies either $Z_{1}\in\mathcal{F}$ or $Z_{2}%
\in\mathcal{F}$.It is clear that every prime $z$-filter is a weakly prime
$z$-filter. But the following example shows that the converse need not be true.

\begin{example}
Consider the ring $C(%
\mathbb{R}
)$ of all real valued continuous functions on $%
\mathbb{R}
.\ $Let $\mathfrak{F}=\{%
\mathbb{R}
\}.\ $Then $\mathcal{F}$ is trivially weakly prime $z$-filter. Take the
following functions:
\[
f(x)=%
\genfrac{\{}{.}{0pt}{}{-\sin x\ ;\ x\in(-\infty,0)}{0\ ;\ x\in\lbrack
0,\infty)}%
\ \text{and }g(x)=%
\genfrac{\{}{.}{0pt}{}{0\ ;\ x\in(-\infty,0)}{\sin x\ ;\ x\in\lbrack0,\infty)}%
.
\]
Since $0=f(x)g(x),\ $we have $z(fg)=z(f)\cup z(g)=%
\mathbb{R}
\in$ $\mathcal{F\ }$but $Z(f)\notin\mathcal{F}$ and $Z(g)\notin\mathcal{F}$.
Thus $\mathcal{F}$\ is not a prime z-filter.
\end{example}

Darani showed that if $P\ $is weakly prime $z$-ideal of $C(X),\ $then
$z[P]\ $is a weakly prime $z$-filter and also if $\mathcal{F}$ is weakly prime
$z$-filter, then $z^{-1}[\mathcal{F}]\ $is a weakly prime $z$-ideal of
$C(X).\ $Now, we will show that any weakly prime $z$-ideal and weakly
1-absorbing $z$-ideals coincide in $C(X).\ $Now, we give the following Lemma
which is needed in the sequel.

\begin{definition}
Let $\mathcal{F}$ be a weakly prime $z$-filter which is not prime $z$-filter.
Then there exists $Z_{1},Z_{2}\in z(X)\ $such that $Z_{1}\cup Z_{2}=X\ $but
$Z_{1}\notin\mathcal{F}$ and $Z_{2}\notin\mathcal{F}$.\ In this case, we say
$(Z_{1},Z_{2})\ $is a double of $\mathcal{F}$.
\end{definition}

\begin{lemma}
Suppose that $\mathcal{F}$ is a weakly prime $z$-filter which is not
prime\ and $(Z_{1},Z_{2})\ $is a double of $\mathcal{F}$.\ Then $\{Z_{1}%
\}\cup\mathcal{F}\ =\{X\}=\ \{Z_{2}\}\cup\mathcal{F}.\ $
\end{lemma}

\begin{proof}
Assume that $(Z_{1},Z_{2})\ $is a double of $\mathcal{F.\ }$Then $Z_{1}\cup
Z_{2}=X$ and $Z_{1},Z_{2}\notin\mathcal{F.\ }$Now, we will show that
$\{Z_{1}\}\cup\mathcal{F}\ =\{X\}.\ $Suppose that $\{Z_{1}\}\cup
\mathcal{F}\ \neq\{X\}.\ $Then there exists $Z\in\mathcal{F}$ such that
$Z_{1}\cup Z\neq X.\ $Then note that $Z_{1}\cup(Z\cap Z_{2})=(Z_{1}\cup
Z)\cap(Z_{1}\cup Z_{2})=(Z_{1}\cup Z)\cap X=(Z_{1}\cup Z)\in\mathcal{F}%
.\ $Since\ $\mathcal{F\ }$is a weakly prime $z$-filter\ and $X\neq Z_{1}%
\cup(Z\cap Z_{2})\in\mathcal{F},\ $we conclude either $Z_{1}\in\mathcal{F\ }%
$or $(Z\cap Z_{2})\in\mathcal{F}.\ $Since $Z\cap Z_{2}\subseteq Z_{2},\ $we
have either $Z_{1}\in\mathcal{F}$ or $Z_{2}\in\mathcal{F}$\ which is contradiction. Thus $\{Z_{1}\}\cup\mathcal{F}\ =\{X\}.\ $Similar argument
shows $\{Z_{2}\}\cup\mathcal{F}\ =\ \{X\}.$
\end{proof}

\begin{theorem}
\label{tf}Every weakly prime $z$-filter $\mathcal{F}$ is either prime
$z$-filter or $\mathcal{F}=\{X\}.\ $
\end{theorem}

\begin{proof}
Suppose that $\mathcal{F\ }$is a weakly prime $z$-filter but not a prime
$z$-filter$.$\ Now, we will show that $\mathcal{F}=\{X\}.\ $Suppose to the
contrary. Then there exists $X\neq Z\in\mathcal{F}$.\ Since $\mathcal{F\ }$is
a weakly prime $z$-filter which is not a prime $z$-filter,\ there exists a
double $(Z_{1},Z_{2})$ of $\mathcal{F}$. By the previous lemma, we conclude that
$Z\cup Z_{1}=Z\cup Z_{2}=X.\ $This implies that $X\neq Z=Z\cap X=Z\cap
(Z_{1}\cup Z_{2})=(Z\cap Z_{1})\cup(Z\cap Z_{2})\in\mathcal{F}$.\ Since
$\mathcal{F\ }$is a weakly prime $z$-filter,\ we conclude that $(Z\cap
Z_{1})\in\mathcal{F}$ or $(Z\cap Z_{2})\in\mathcal{F,}$ and thus $Z_{1}%
\in\mathcal{F}$ or $Z_{2}\in\mathcal{F}$\ which are contradictions. Hence
$\mathcal{F}=\{X\}.\ $
\end{proof}

\begin{theorem}
\label{tf2}Suppose that $P\ $is a proper ideal of $C(X)\ $and $\mathcal{F}$ is
a $z$-filter on $X.\ $Then, we have the followings.

(i) \ If $P\ $is a weakly prime ideal of $C(X),\ $then $z[P]=\{X\}\ $or
$z[P]\ $is a prime $z$-filter on $X.$

(ii)\ If $\mathcal{F}$ is a weakly prime $z$-filter on $X,$ then
$z^{-1}[\mathcal{F}]\ $is a prime $z$-ideal or $0=z^{-1}[\mathcal{F}].\ $
\end{theorem}

\begin{proof}
(i):\ Suppose that $P\ $is a weakly prime ideal of $C(X).\ $Since $C(X)\ $is
reduced ring, $P=0\ $or $P\ $is a prime ideal of $C(X).\ $If $P=0,\ $then
$z[P]=\{X\}$. Now assume that $P\ $is a prime ideal of $C(X).\ $Since
$z^{-1}[z[P]]\ $contains the prime ideal $P,\ z^{-1}[z[P]]$ is a prime
$z$-ideal and so $z[P]\ $is a prime $z$-filter on $X$ by \cite{GilJe}$.\ $

(ii):\ Suppose that $\mathcal{F}$ is a weakly prime $z$-filter on $X.\ $Then
by Theorem \ref{tf}, $\mathcal{F}$ is either prime $z$-filter or
$\mathcal{F}=\{X\}.\ $If $\mathcal{F}=\{X\},\ $then $0=z^{-1}[\mathcal{F}].$
So assume that $\mathcal{F}$ is a prime $z$-filter. Then by \cite{GilJe}%
,\ $z^{-1}[\mathcal{F}]\ $is a prime $z$-ideal,as desired.
\end{proof}

\begin{theorem}
\label{tcon}Suppose that $P\ $is a proper ideal of $C(X).\ $Then the following
statements hold.

(i)\ If $P\ $is a weakly 1-absorbing prime $z$-ideal, then $z[P]\ $is a weakly
prime $z$-filter.

(ii)\ If $\mathcal{F}$ is a weakly prime $z$-filter on $X,\ $then
$z^{-1}[\mathcal{F}]$ is a weakly 1-absorbing prime $z$-ideal.

(iii)\ $P\ $is a weakly 1-absorbing prime $z$-ideal if and only if $P\ $is a
weakly prime $z$-ideal if and only if either $P=(0)\ $or $P\ $is a prime $z$-ideal.
\end{theorem}

\begin{proof}
(i):\ Suppose that $P\ $is a weakly 1-absorbing prime $z$-ideal.\ Since
$P\ $is a $z$-ideal, we conclude that $P=\sqrt{P}.\ $As $P\ $is a weakly
1-absorbing prime ideal and $C(X)\ $is a reduced ring, by Theorem \ref{tred},
$P\ $is a weakly prime $z$-ideal of $C(X).\ $Then by Theorem \ref{tf2}, we
conclude that $z[P]\ $is a weakly prime $z$-filter.

(ii):\ Suppose that $\mathcal{F}$ is a weakly prime $z$-filter on $X.\ $Then
by Theorem \ref{tf2}, $z^{-1}[\mathcal{F}]\ $is a prime $z$-ideal or
$0=z^{-1}[\mathcal{F}].\ $In both cases, $z^{-1}[\mathcal{F}]$ is a weakly
1-absorbing prime $z$-ideal.

(iii):\ The implication $P=(0)\ $or $P\ $is a prime $z$-ideal $\Rightarrow$
$P\ $is a weakly prime $z$-ideal $\Rightarrow$\ $P\ $is a weakly 1-absorbing
prime $z$-ideal is clear. Now, let $P\ $be a weakly 1-absorbing prime
$z$-ideal of $C(X).\ $Then by (i),\ $z[P]\ $is a weakly prime $z$-filter. Then
$z^{-1}[z[P]]=P$ is a prime $z$-ideal or $P=(0)\ $by Theorem \ref{tf2}.\ 
\end{proof}


\begin{thebibliography}{99}                                                                                               %


\bibitem {AnSmi}D. D. Anderson and E. Smith, {\it Weakly prime ideals},
Houston J. Math. \textbf{29} (2003), no. 4, 831-840.

\bibitem {AnWi}D. D. Anderson and M. Winders, {\it Idealization of a
module}, Journal of Commutative Algebra, \textbf{1} (2009), no. 1, 3-56.

\bibitem {AtMac}M. Atiyah and I. MacDonald, {\it Introduction to
commutative algebra}, Avalon Publishing, 1994.

\bibitem {Be}C. Beddani and W. Messirdi, {\it 2-prime ideals and their
applications}, Journal of Algebra and Its Applications, \textbf{15} (2016), no. 13, 1650051.

\bibitem {Ba1}A. Badawi, {\it On 2-absorbing ideals of commutative rings},
Bulletin of the Australian Mathematical Society, \textbf{75} (2007), no. 3, 417-429.

\bibitem {Ba2}A. Badawi, {\it On weakly semiprime ideals of commutative
rings}, Beitr\"{a}ge zur Algebra und Geometrie/Contributions to Algebra and
Geometry, \textbf{57} (2016), no. 3, 589-597.

\bibitem {BaDa}A. Badawi and A. Y. Darani, {\it On weakly 2-absorbing
ideals of commutative rings}, Houston J. Math, \textbf{39} (2013), no. 2, 441-452.

\bibitem {BaTeYe}A. Badawi, \"U. Tekir, and E. Yetkin, {\it On 2-absorbing
primary ideals in commutative rings}, Bulletin of the Korean Mathematical
Society, \textbf{51} (2014), no. 4, 1163-1173.

\bibitem {BilDemOr}Z. Bilgin, E. Demir, and K. H. Oral, {\it On 2-absorbing
z-filters}, Bull. Math. Soc. Sci. Math. Roumanie, \textbf{61(109)} (2018), No. 2, 147-155.

\bibitem {Dar}A. Y. Darani, {\it Primal and weakly primal ideals in C(X)},
Palestine Journal of Mathematics,  \textbf{6} (Special Issue: II) (2017), 117-122.

\bibitem {GilJe}L. Gillman and M. Jerison, {\it Rings of continuous
functions}, Courier Dover Publications, 2017.

\bibitem {KoUlTe}S. Koc, \"U. Tekir, and G. Ulucak, {\it On strongly quasi
primary ideals}, Bulletin of Korean Mathematical Society, \textbf{56} (2019), no. 3, 729-743.

\bibitem {KoUrTe}S. Koc, R. N. Uregen, and  \"U. Tekir, {\it On 2-absorbing
quasi primary submodules}, Filomat, \textbf{31} (2017), no. 10, 2943-2950.

\bibitem {Na}M. Nagata, {\it Local rings}, Interscience Tracts in Pure and
Appl. Math., 1962.

\bibitem {Sharp}R. Y. Sharp, {\it Steps in commutative algebra}, (No. 51),
Cambridge university press, 2000.

\bibitem {TeKoOrSh} \"U. Tekir, S. Ko\c{c}, K. H. Oral, and K. P. Shum, {\it On 2-absorbing quasi-primary ideals in commutative rings}, Communications in
Mathematics and Statistics, \textbf{4} (2016), no. 1, 55-62.

\bibitem {YasNik}A. Yassine, M. J. Nikmehr, and R. Nikandish,  {\it On
1-absorbing prime ideals of commutative rings}, Journal of Algebra and Its Applications, (2020), (Accepted).
\end{thebibliography}
\end{document}